\documentclass[11pt]{article}
\usepackage{amsfonts}
\usepackage{latexsym,amsmath,amsthm,amssymb,color}
\usepackage{mathrsfs,wasysym,graphicx,lscape}

\begin{document}
\newtheorem{theorem}{\indent Theorem}[section]
\newtheorem{proposition}[theorem]{\indent Proposition}
\newtheorem{definition}[theorem]{\indent Definition}
\newtheorem{lemma}[theorem]{\indent Lemma}
\newtheorem{remark}[theorem]{\indent Remark}
\newtheorem{corollary}[theorem]{\indent Corollary}

%%%%%%%%%%
\begin{center}
    {\large \bf Normalized ground states for the fractional nonlinear  Schr\"{o}dinger equations}
\vspace{0.5cm}\\{Binhua Feng$^{1,*}$, Jiajia Ren$^{1}$, Qingxuan Wang$^{2,*}$}\\
{\small $^1$Department of Mathematics, Northwest Normal University, Lanzhou, 730070, China \\$^2$Department of Mathematics, Zhejiang Normal University, Jinhua, 321004, China}\\
\end{center}

\renewcommand{\theequation}{\arabic{section}.\arabic{equation}}
\numberwithin{equation}{section}
\footnote[0]{\hspace*{-7.4mm}
%%%%%%%%%%
E-mail: binhuaf@163.com(B. Feng), renjiajia0327@163.com(J. Ren), wangqx@zjnu.edu.cn(Q. Wang)\\
$^*$Corresponding author.
}

\renewcommand{\baselinestretch}{1.3}
\large\normalsize
\begin{abstract}
In this paper, we study the existence and instability of standing waves with a prescribed $L^2$-norm for the fractional Schr\"{o}dinger equation
\begin{equation}\label{e0}
	i\partial_{t}\psi=(-\Delta)^{s}\psi-f(\psi),
	\end{equation}
	where $0<s<1$, $f(\psi)=|\psi|^{p}\psi$ with $\frac{4s}{N}<p<\frac{4s}{N-2s}$ or $f(\psi)=(|x|^{-\gamma}\ast|\psi|^2)\psi$ with $2s<\gamma<\min\{N,4s\}$. To this end, we look for normalized solutions of the associated stationary equation
\begin{equation}\label{elliptic equation0}
(-\Delta)^s u+\omega u-f(u)=0.
\end{equation}
Firstly, by constructing a suitable submanifold of a $L^2$-sphere, we prove the existence of a normalized solution for \eqref{elliptic equation0} with least energy in the $L^2$-sphere, which corresponds to a normalized ground state standing wave of \eqref{e0}. Then, we show that each normalized ground state of \eqref{elliptic equation0} coincides a ground state of \eqref{elliptic equation0} in the usual sense. Finally, we obtain the sharp threshold of global
existence and blow-up for \eqref{e0}. Moreover, we can use this sharp threshold to show that all normalized ground state standing waves are strongly unstable by blow-up.

\end{abstract}
\section{Introduction}
	In recent years, there has been a great deal of interest in using fractional Laplacians to model the physical phenomena. By extending the
	Feynman path integral from the Brownian-like to the L\'evy-like quantum mechanical
	paths, Laskin in \cite{Laskin1,Laskin2} used the theory of functionals over functional measure generated
	by the L\'evy stochastic process to deduce the following fractional nonlinear Schr\"{o}dinger	equation (NLS)
	\begin{equation}\label{e}
	i\partial_{t}\psi=(-\Delta)^{s}\psi-f(\psi),~~~\psi(0,x) = \psi_0 (x),
	\end{equation}
	where $0<s<1$, $f(\psi)=|\psi|^{p}\psi$ or $f(\psi)=(|x|^{-\gamma}\ast|\psi|^2)\psi$. The fractional differential
	operator $(-\Delta)^{s}$ is defined by $(-\Delta)^{s}\psi=\mathcal{F}^{-1}[|\xi|^{2s}\mathcal{F}(\psi)]$, where $\mathcal{F}$ and $\mathcal{F}^{-1}$
	are the Fourier transform and inverse Fourier transform, respectively. The fractional NLS also appears in the continuum limit of discrete models with long-range interactions (see e.g. \cite{KirkpatrickLenzmannStaffilani}) and in the description of Boson stars as well as in water wave dynamics (see e.g. \cite{FrohlichJonssonLenzmann}).

%In this paper, we consider the orbital stability and strong instability of standing waves for the fractional NLS with combined power-type nonlinearities
%\begin{equation}\label{FNLS power}
%\left\{
%\begin{array}{l}
%i\partial_t\psi- (-\Delta)^s \psi+|\psi|^{p}\psi=0,~~~~
%(t,x)\in [0,T^*)\times \mathbb{R}^{N}, \\
%\psi(0,x) = \psi_0 (x) ,%
%\end{array}%
%\right.
%\end{equation}%
%and
%\begin{equation}\label{FNLS Hartree}
%\left\{
%\begin{array}{l}
%i\partial_t\psi- (-\Delta)^s \psi+(|x|^{-\lambda}\ast|\psi|^2)\psi=0,~~~~
%(t,x)\in [0,T^*)\times \mathbb{R}^{N}, \\
%\psi(0,x) = \psi_0 (x) ,%
%\end{array}%
%\right.
%\end{equation}%
%where $\psi:[0,T^*)\times \mathbb{R}^{N}\rightarrow \mathbb{C}$ is the complex valued function, $N\geq 2$, $\psi_0 \in H^s$, $0<s<1$, $0<T^*\leq \infty$, $0<p_1<p_2<\frac{4s}{N-2s}$.

The intention of this paper is to study \eqref{e} from a variational perspective. To
this end, it is of great interest to consider standing waves to \eqref{e}, which are solutions
of the form $e^{i\omega t}u$, where $\omega \in \mathbb{R}$ is
a frequency and $u$ is complex-valued. This ansatz yields
\begin{equation}\label{elliptic equation}
(-\Delta)^s u+\omega u-f(u)=0,
\end{equation}
where $f(u)=|u|^{p}u$ or $f(u)=(|x|^{-\gamma}\ast|u|^2)u$.

At this moment, our intention is reduced to explore \eqref{elliptic equation}. To do this, we would
like to mention two substantially distinct options in terms of the frequency $\omega$. The first one is
to fix the frequency $\omega\in \mathbb{R}$. In this situation, every solution to \eqref{elliptic equation} corresponds to a
critical point of the action functional $J(u)$ on $H^s$, where
\begin{align}\label{action power}
J(u):= \frac{1}{2}\|u\|_{\dot{H}^s}^2+\frac{\omega}{2}\|u\|_{L^2}^2-\frac{1}{p+2}\|u\|_{L^{p+2}}^{p+2},~~if~f(u)=|u|^{p}u,
\end{align}
\begin{align}\label{action Hartree}
J(u):= \frac{1}{2}\|u\|_{\dot{H}^s}^2+\frac{\omega}{2}\|u\|_{L^2}^2-\frac{1}{4}\int_{\mathbb{R}^N} (|x|^{-\gamma}\ast|u|^2)(x)|u(x)|^2 dx,~~if~f(u)=(|x|^{-\gamma}\ast|u|^2)u.
\end{align}
In this case particular attention is devoted to least action solutions, namely solutions minimizing
$J(u)$ among all non-trivial solutions.

Alternatively, it is interesting to study solution to \eqref{elliptic equation} having prescribed $L^2$-norm, namely, for any given $c>0$, to consider solution to \eqref{elliptic equation} satisfying the
$L^2$-norm constraint
\begin{align}\label{mass constraint}
S(c)=\{u\in H^s:~~\|u\|_{L^2}^2=c\},~~c>0.
\end{align}
Physically, such a solution is so-called normalized solution to \eqref{elliptic equation}, which formally
corresponds to a critical point of the energy functional $E(u)$ restricted on $S(c)$, where
\begin{align}\label{energy power}
E(u):= \frac{1}{2}\|u\|_{\dot{H}^s}^2-\frac{1}{p+2}\|u\|_{L^{p+2}}^{p+2},~~if~f(u)=|u|^{p}u,
\end{align}
\begin{align}\label{energy Hartree}
E(u):= \frac{1}{2}\|u\|_{\dot{H}^s}^2-\frac{1}{4}\int_{\mathbb{R}^N} (|x|^{-\gamma}\ast|u|^2)(x)|u(x)|^2 dx,~~if~f(u)=(|x|^{-\gamma}\ast|u|^2)u.
\end{align}
It is worth pointing out that, in this situation, the frequency $\omega\in \mathbb{R}$ is an unknown part, which
is determined as the Lagrange multiplier associated to the constraint $S(c)$.

From a physical point of view, it is quite meaningful to consider normalized solution to \eqref{elliptic equation}. This is not only because the $L^2$-norm of solution to the Cauchy problem of \eqref{e} is conserved along time, that is, for any $t>0$
\[
\int_{\mathbb{R}^N} |\psi(t,x)|^2 dx=\int_{\mathbb{R}^N} |\psi_0(x)|^2 dx,
\]
see Proposition \ref{proposition radial lwp}, but also because the mass has often a clear physical meaning; for instance, it represents the power supply
in nonlinear optics, or the total number of atoms in Bose-Einstein condensation, two main fields
of application of the NLS.
Moreover, this approach turns out to be useful also from the purely
mathematical perspective, since it gives a better insight of the properties of the stationary solutions
for \eqref{e}, such as stability or instability (this was already evident in the seminal contributions by
H. Berestycki and T. Cazenave \cite{bcaze}, and by T. Cazenave and P.-L. Lions \cite{cl}). For these reasons,
here we focus on existence and properties of solutions to \eqref{elliptic equation} with prescribed mass and the $L^2$-supercritical nonlinearity, a problem
which was, up to now, essentially unexplored.

The existence of normalized stationary states can be formulated as the following problem: given
$c>0$, we aim to find $(u_c,\omega_c)\in H^s\times \mathbb{R}$ solving \eqref{elliptic equation} together
with the normalization condition \eqref{mass constraint}.
When $f(u)=|u|^{p}u$ with $0<p<\frac{4s}{N-2s}$ or $f(u)=(|x|^{-\gamma}\ast|u|^2)u$ with $0<\gamma<\min\{N,4s\}$,
 it is standard that $E(u)$ is of class $C^1$ in $H^s$, and any critical point $u$
of $E|_{S(c)}$ corresponds to a solution to \eqref{elliptic equation} satisfying \eqref{mass constraint}, with the parameter $\omega\in \mathbb{R}$ appearing
as Lagrange multiplier. We are particularly interested in ground state solutions, defined as
follows:
 \begin{definition}
(Ground states). We write that $u_c$ is a ground state of \eqref{elliptic equation} on $S(c)$ if it is a solution to \eqref{elliptic equation} having
minimal energy among all the solutions which belongs to $S(c)$:
\[
E'|_{S(c)}(u_c)=0~~and~~E(u_c)=\inf\{E(v_c):~v_c\in S(c),~E'|_{S(c)}(v_c)=0\}.
\]
%The set of the ground states will be denoted by $\mathcal{M}_c$.
\end{definition}
Before stating our main results, let us recall known results related to the normalized solutions for some Schr\"{o}dinger type equations and systems.
It is well known that, when
dealing with the Schr\"{o}dinger equations, the $L^2$-critical exponent
plays a special role. This is the threshold exponent for many dynamical properties such as global
existence vs. blow-up, and the stability or instability of ground states. From the variational point
of view, if the problem is purely $L^2$-subcritical, then $E(u)$ is bounded from
below on $S(c)$. Thus, for every $c>0$, ground states can be found as global minimizers of $E|_{S(c)}$, see \cite{ca2003,cl}.  Moreover, the set of ground states is orbitally stable. In the $L^2$-supercritical case, on the contrary, $E|_{S(c)}$ is unbounded from below. By
exploiting the mountain pass lemma and a smart compactness argument, L. Jeanjean \cite{jean97NA} showed that a normalized ground state does exist for every $c>0$ also in this case. For quite a long time the paper \cite{jean97NA} was the only one dealing with existence of normalized
solutions in cases when the energy is unbounded from below on the $L^2$-constraint. More recently,
however, problems of this type received much attention, see \cite{jeansiam,jean-luo,jean-gou Tran,jean-luo-wang,luoxiao,soave,soave critical} for normalized solutions
to scalar equations in the whole space $\mathbb{R}^N$, see \cite{bartsch-Jeanjean,bartsch-Jeanjean-Soave,soajfa,bartsch-Soave,
gou-zhang,gou-nonlinearity,li-luo} for normalized solutions to systems in $\mathbb{R}^N$.

For the fractional Schr\"{o}dinger equation \eqref{elliptic equation}, in the $L^2$-subcritical case, i.e. $0<p<\frac{4s}{N}$ or $0<\gamma<2s$, $E(u)$ is bounded from
below on $S(c)$. Thus, for every $c>0$, ground states can be found as global minimizers of $E|_{S(c)}$. Moreover, the set of ground states is orbitally stable. Recently, these problems have been studied by using
the concentration compactness principle in \cite{bjde,chcpaa,f13ejde,f19jmp,guobl,wud}, using
 the profile decomposition theory in \cite{f18cma,f18jmaa,pengshi,zhujdde,zjee}. In the $L^2$-supercritical case, i.e. $\frac{4s}{N}<p<\frac{4s}{N-2s}$ or $2s<\gamma<2s$, on the contrary, $E|_{S(c)}$ is unbounded from below.
 To the best of our knowledge, there are no any results in this case.
%The standing waves of \eqref{FNLS} are solutions of the form
%We say that $u\in H^s$ is a weak solution to problem \eqref{elliptic equation} if
%\[
%\int_{\mathbb{R}^N}(-\Delta)^{s/2} u(-\Delta)^{s/2}\varphi dx+\omega\int_{\mathbb{R}^N} u\varphi dx-\int_{\mathbb{R}^N} |u|^{p}u\varphi dx=0,
%\]
%and
%\[
%\int_{\mathbb{R}^N}(-\Delta)^{s/2} u(-\Delta)^{s/2}\varphi dx+\omega\int_{\mathbb{R}^N} u\varphi dx-\int_{\mathbb{R}^N} (|x|^{-\lambda}\ast|u|^2)u\varphi dx=0.
%\]
%for all $\varphi \in C_0^\infty(\mathbb{R}^N)$ and $(u,\mu)\in H^s\times \mathbb{R}$ is a couple of weak solution to problem \eqref{elliptic equation} if $u$ is a weak solution to \eqref{elliptic equation}  with $\omega=\mu$.

The aim of this paper is to consider the existence and properties of normalized
ground states to \eqref{elliptic equation}, the sharp threshold of global
existence and blow-up, and the strong instability of normalized
ground state standing waves for \eqref{e} in the $L^2$-supercritical case.
Our main results are as follows:
\begin{theorem}\label{theorem existence}
Let $f(u)=|u|^{p}u$ with $\frac{4s}{N}<p<\frac{4s}{N-2s}$ or $f(u)=(|x|^{-\gamma}\ast|u|^2)u$ with $2s<\gamma<\min\{N,4s\}$. Then for any $c>0$, there exists a couple of weak solution $(u_c,\omega_c)\in H^s\times \mathbb{R}^+$ to problems \eqref{elliptic equation}-\eqref{mass constraint}. Moreover, we have
\begin{equation*}
\left\{
\begin{array}{l}
\|u_c\|_{\dot{H}^s}\rightarrow +\infty, \\
\omega_c\rightarrow +\infty,\\
E(u_c)\rightarrow +\infty,
\end{array}%
\right.
\end{equation*}%
as $c\rightarrow 0^+$ and
\begin{equation*}\label{FNLS Hartree}
\left\{
\begin{array}{l}
\|u_c\|_{\dot{H}^s}\rightarrow 0, \\
\omega_c\rightarrow 0,\\
E(u_c)\rightarrow 0,
\end{array}%
\right.
\end{equation*}%
as $c\rightarrow +\infty$.
\end{theorem}
To the best of our knowledge, this seems to be the first
contribution regarding existence of normalized
ground states for the fractional NLS in the $L^2$-supercritical case.
The proof of this theorem is based on a constrained minimization method. In the mass-supercritical case, i.e., $\frac{4s}{N}<p<\frac{4s}{N-2s}$ or $2s<\gamma<\min\{N,4s\}$, the functional $E(u)$ is no longer bounded from below on $S(c)$, the minimization method on $S(c)$ used in \cite{f18jmaa,pengshi,zhujdde} does not work. Motivated by minimization method on Pohozaev manifold, we try to construct a submanifold of $S(c)$, on which $E(u)$ is bounded from below and coercive, and then we look for minimizers of $E(u)$ on such a submanifold.
Precisely, we introduce an auxiliary functional
\begin{align}\label{qupower}
Q(u):=s\|u\|_{\dot{H}^s}^2 -\frac{Np}{2(p+2)}\|u\|_{L^{p+2}}^{p+2},~~if~f(u)=|u|^{p}u,
\end{align}
\begin{align}\label{quhartree}
Q(u):=s\|u\|_{\dot{H}^s}^2 -\frac{\gamma}{4}\int_{\mathbb{R}^N} (|x|^{-\gamma}\ast|u|^2)(x)|u(x)|^2 dx,~~if~f(u)=(|x|^{-\gamma}\ast|u|^2)u.
\end{align}
and construct a submanifold $V(c)$ as follows
\begin{align}\label{submanifold}
V(c):=\{u\in S(c):~~Q(u)=0\}.
\end{align}
By considering the minimization problem
\begin{align}\label{minimization problem submanifold}
m(c):=\inf_{u\in V(c)}E(u),
\end{align}
we find a critical point of $E$ restricted to $V(c)$ and prove that it is indeed a critical point of $E$ restricted to $S(c)$. Let us denote the set of minimizers of $E$ on $V(c)$ as
\begin{align}\label{set minimizers}
\mathcal{M}_c:=\{u\in V(c):~E(u)=\inf_{v\in V(c)}E(v)\}.
\end{align}
Then we prove the existence part of Theorem \ref{theorem existence} by showing a simple property of $\mathcal{M}_c$.

Compared with \cite{jean97NA}, we use a constrained minimization method instead of a mini-max procedure. Although these two methods both work on finding a normalized ground state of problem \eqref{elliptic equation}-\eqref{mass constraint}, we believe that the constrained minimization method is more convenient in getting the normalized ground state solution to problem \eqref{elliptic equation}-\eqref{mass constraint}.
In particular, in order to solve the minimization problem \eqref{minimization problem submanifold}, we consider an equivalent minimization problem \eqref{minimization1}, which can be easily solved by using Brezis-Lieb's lemma. Moreover, it is easier to obtain the sharp threshold of global
existence and blow-up for \eqref{e} by using the minimization problem \eqref{minimization1}.

For any $\omega>0$, the existence of ground state solution $u_\omega$ to problem \eqref{elliptic equation} has been studied in \cite{bou,dinhinstability,gzjde,saanounijmp,zhujdde}. Next, we analyze the connection between the couple of weak solution $(u_c,\omega_c)$ to \eqref{elliptic equation} obtained in Theorem \ref{theorem existence} and $u_{\omega_c}$.
\begin{theorem}\label{theorem ground states}
Let $f(u)=|u|^{p}u$ with $\frac{4s}{N}<p<\frac{4s}{N-2s}$ or $f(u)=(|x|^{-\gamma}\ast|u|^2)u$ with $2s<\gamma<\min\{N,4s\}$. Then for any $u_c\in \mathcal{M}_c$, there exists $\omega_c>0$ such that $(u_c,\omega_c)\in H^s\times \mathbb{R}$ is a couple of weak solution to problem \eqref{elliptic equation}.
Furthermore, $u_c$ is a ground state solution to problem \eqref{elliptic equation} with $\omega=\omega_c$.
\end{theorem}
\textbf{Remark.} $u_c$ is a ground state solution to problem \eqref{elliptic equation} with $\omega=\omega_c$ means that
\[
J(u_c)=\inf\{J(u):~~J'(u)=0~and~u\neq 0\},
\]
where $J(u)$ is defined in \eqref{action power} or \eqref{action Hartree}. Theorem \ref{theorem ground states} indicates that every normalized ground state of problem \eqref{elliptic equation} coincides a ground state of problem \eqref{elliptic equation}. This information is interesting itself. For example, it is well-known that the ground state solution $u_\omega$ of \eqref{elliptic equation} with $f(u)=|u|^{p}u$ is unique up to translation, see \cite{uniqueness1,uniqueness2}. We consequently obtain that for every $c>0$, the solution of minimizing problem \eqref{minimization problem submanifold} is unique up to translation.  Moreover, based on the minimizing problems \eqref{minimization problem submanifold} and \eqref{minimization1},  To this end, we introduce the following invariant sets.
\[
\mathcal{A}_c:=\{u\in S(c):~~Q(u)>0~and~E(u)<m(c)\},
\]
\[
\mathcal{B}_c:=\{u\in S(c):~~Q(u)<0~and~E(u)<m(c)\}.
\]
\begin{theorem}\label{theorem sharp}(Global versus blow-up dichotomy)
Let $N\geq 2$, $\frac{N}{2N-1} \leq s <1$, $\psi_0\in H^s$, $f(\psi)=|\psi|^{p}\psi$ with $\frac{4s}{N}<p<\frac{4s}{N-2s}$, or $f(\psi)=(|x|^{-\gamma}\ast|\psi|^2)\psi$ with $2s<\gamma<\min\{N,4s\}$.
Then, $\mathcal{A}_{\|\psi_0\|_{L^2}^2}$ and $\mathcal{B}_{\|\psi_0\|_{L^2}^2}$ are two invariant manifolds of \eqref{e}. More precisely, if $\psi_0\in \mathcal{A}_{\|\psi_0\|_{L^2}^2}$ or $\psi_0\in \mathcal{B}_{\|\psi_0\|_{L^2}^2}$, then the solution $\psi(t)$ satisfies $\psi(t)\in \mathcal{A}_{\|\psi_0\|_{L^2}^2}$ or $\psi(t)\in \mathcal{B}_{\|\psi_0\|_{L^2}^2}$ for any $t\in [0,T^*)$, respectively. Moreover, we can obtain the following sharp threshold of global
existence and blow-up for \eqref{e}.

(1) If $\psi_0\in \mathcal{A}_{\|\psi_0\|_{L^2}^2}$, then the solution $\psi(t)$ of \eqref{e} with initial data $\psi_0$ exists globally in time.

(2) When $f(\psi)=|\psi|^{p}\psi$, assume further that $p<4s$, $\psi_0 \in \mathcal{B}_{\|\psi_0\|_{L^2}^2}$ and $\psi_0$ is radial, then the solution $\psi(t)$ of \eqref{e} blows up in finite time.

(3) When $f(\psi)=(|x|^{-\gamma}\ast|\psi|^2)\psi$, assume further that $x\psi_0\in H^{s_0}$ with $s_0=\max\{2s,\frac{\gamma+1}{2}\}$, $x\psi_0\in L^2$, $x\cdot\nabla \psi_0\in L^2$, $\psi_0 \in \mathcal{B}_{\|\psi_0\|_{L^2}^2}$ and $\psi_0$ is radial, then the solution $\psi(t)$ of \eqref{e} blows up in finite time.
\end{theorem}
\textbf{Remark 1.} Note that the condition $p<4s$ is technical due to the localized virial estimate, see Lemma \ref{lemma Hs radial virial estimate}. However, this only leads to a restriction in the two dimensional case. Indeed, for $N\geq 3$ and $2s<2$, we have $p<\frac{4s}{N-2s}<4s$.

\textbf{Remark 2.}
For the classical NLS, i.e., $s=1$ in \eqref{e}, it follows from  the virial identity and \eqref{Qxiaoyuling} with $s=1$ that
\[
\frac{d^2}{dt^2}\|x\psi(t)\|_{L^2}^2=4Q(\psi(t))\leq 8(E(\psi_0)-m(\|\psi_0\|_{L^2}^2))<0,
\]
where $Q(\psi(t))$ is defined by \eqref{qupower} or \eqref{quhartree} with $s=1$.
This implies that the solution $\psi(t)$ of \eqref{e} with $s=1$ blows up in finite time.

But for the fractional NLS \eqref{e} with $f(\psi)=|\psi|^{p}\psi$, it follows from Lemma \ref{lemma Hs radial virial estimate} and \eqref{Qxiaoyuling} that
\begin{align*}\label{local}
		\frac{d}{dt} M_{\varphi_R}(\psi(t))&\leq 4Q(\psi(t))+ C\eta\|\psi(t)\|^2_{\dot{H}^s}+\circ_R(1)\\&\leq 8s(E(\psi_0)-m(\|\psi_0\|_{L^2}^2))+ C\eta\|\psi(t)\|^2_{\dot{H}^s}+\circ_R(1)\\&\leq C\eta\|\psi(t)\|^2_{\dot{H}^s}+\circ_R(1),
		\end{align*}
where $\eta>0$, $\circ_R(1)\rightarrow 0$ as $R\rightarrow \infty$, $\|\psi(t)\|^2_{\dot{H}^s}$ may be unbounded.
Therefore, there exist some essential difficulties in proving Theorem \ref{theorem sharp} between the fractional NLS and the classical NLS.
In this paper, we will develop some new ideas to solve these problems.

Notice that $\mathcal{B}_c$ contains functions arbitrary close to $u_c$ in $H^s$. Indeed, letting $u_c^\lambda(x)=\lambda^{N/2}u_c(\lambda x)$ with $\lambda>1$, it easily follows that $u_c^\lambda \in \mathcal{B}_c$ and $u_c^\lambda \rightarrow u_c$ as $\lambda \rightarrow 1$.
Therefore, as an immediate corollary of Theorem \ref{theorem sharp}, we can derive
the strong instability of normalized ground states to \eqref{e}.
\begin{corollary}
Let $N\geq 2$, $\frac{N}{2N-1} \leq s <1$, $c>0$, $f(u)=|u|^{p}u$ with $\frac{4s}{N}<p<\frac{4s}{N-2s}$ and $p<4s$. Assume that $u_c\in \mathcal{M}_c$, the standing wave $\psi(t,x)=e^{i\omega_c t}u_c(x)$ is
strongly unstable in the following sense: there exists $\{\psi_{0,n}\}\subset H^s$ such that
$\psi_{0,n}\rightarrow u$ in $H^s$ as $n \rightarrow \infty$ and the corresponding solution $\psi_n$ of \eqref{e} with initial data $\psi_{0,n}$
blows up in finite time for any $n\geq 1$.
\end{corollary}
\begin{corollary}
Let $N\geq 2$, $\frac{N}{2N-1} \leq s <1$, $c>0$, $f(u)=(|x|^{-\gamma}\ast|u|^2)u$ with $2s<\gamma<\min\{N,4s\}$. Then for any $u_c\in \mathcal{M}_c$ radial, the standing wave $\psi(t,x)=e^{i\omega_c t}u_c(x)$ is
strongly unstable in the following sense: there exists $\{\psi_{0,n}\}\subset H^s$ such that
$\psi_{0,n}\rightarrow u$ in $H^s$ as $n \rightarrow \infty$ and the corresponding solution $\psi_n$ of \eqref{e} with initial data $\psi_{0,n}$
blows up in finite time for any $n\geq 1$.
\end{corollary}
\textbf{Remark.} It is well-known that the ground state solution $u_\omega$ of \eqref{elliptic equation} with $f(u)=|u|^{p}u$ is unique up to translation, see \cite{uniqueness1,uniqueness2}. Based on this fact and the translation invariance of \eqref{e}, we can prove that for every $u_c\in \mathcal{M}_c$, the ground state standing wave $\psi(t,x)=e^{i\omega t}u_\omega(x)$ is strongly unstable.
But, to the best of our knowledge, the uniqueness of ground state solution $u_\omega$ of \eqref{elliptic equation} with $f(u)=(|x|^{-\gamma}\ast|u|^2)u$ is unknown, so we only prove the instability of radial normalized ground states.

This paper is organized as follows: in Section 2, we firstly collect
some lemmas such as the local well-posedness theory of \eqref{e}, Brezis-Lieb's lemma, a compactness lemma, a sharp Gagliardo-Nirenberg type inequality and the localized virial estimate related to \eqref{e}. In section 3, 4 and 5, we will prove Theorems \ref{theorem existence},  \ref{theorem ground states} and \ref{theorem sharp} respectively.

{\textbf {Notations.}} Throughout this paper, we use the following
notations. $C> 0$ will stand for a constant that may be different from
line to line when it does not cause any confusion.  For any $s
\in(0,1)$, the fractional Sobolev space $H^s(\mathbb{R}^N)$ is defined by
\[
H^s(\mathbb{R}^N)=\left\{u\in
L^2(\mathbb{R}^N);~~\int_{\mathbb{R}^N}(1+|\xi|^{2s})|\hat{u}(\xi)|^2d\xi<\infty
\right\},
\]
endowed with the norm
\[
\|u\|_{H^s(\mathbb{R}^N)}=\|u\|_{L^2(\mathbb{R}^N)}+\|u\|_{\dot{H}^s(\mathbb{R}^N)},
\]
where up to a multiplicative constant
\[
\|u\|_{\dot{H}^s(\mathbb{R}^N)}=\left\{\int\!\!\!\!\int_{\mathbb{R}^N\times
\mathbb{R}^N
}\frac{|u(x)-u(y)|^2}{|x-y|^{N+2\alpha}}dxdy\right\}^{\frac{1}{2}}
\]
is the so-called Gagliardo semi-norm of $u$.
In this paper, we often use the abbreviations
$L^{r}=L^{r}(\mathbb{R}^N)$, $H^s=H^s(\mathbb{R}^N)$.

\section{Preliminaries}
In this section, we recall some preliminary results that
will be used later.
Firstly, let us recall the local theory for the Cauchy problem \eqref{e}.
The local well-posedness for \eqref{e} in the energy space $H^s$ was first studied by Hong and Sire in \cite{hong}. The proof is based on Strichartz estimates and the contraction mapping argument. Note that
for non-radial data, Strichartz estimates have a loss of derivatives. Fortunately, this loss of derivatives can
be compensated by using Sobolev embedding. However, it leads to a weak local well-posedness in the
energy space compared to the classical nonlinear Schr\"{o}dinger equation. We refer the reader
to \cite{dinhlocal,hong} for more details. One can remove the loss of derivatives in Strichartz estimates by considering
radially symmetric data. However, it needs a restriction on the validity of $s$, namely $\frac{N}{2N-1} \leq s <1$. More
precisely, we have the following local well-posedness for \eqref{e} with radial $H^s$ initial data established in \cite{f19jmp}.
	\begin{proposition}[Radial $H^s$ LWP] \label{proposition radial lwp}
		Let $N\geq 2$, $\frac{N}{2N-1} \leq s <1$, $f(\psi)=|\psi|^{p}\psi$ with $\frac{4s}{N}<p<\frac{4s}{N-2s}$ or $f(\psi)=(|x|^{-\gamma}\ast|\psi|^2)\psi$ with $2s<\gamma<\min\{N,4s\}$.
		Then for any $\psi_0 \in H^s$ radial, there exist $T^*\in(0,+\infty]$ and a unique solution to $(\ref{e})$ satisfying $\psi \in C([0,T^*), H^s)$.
		Moreover, the following properties hold:
		\begin{itemize}
			\item $\psi \in L^a_{\emph{loc}}([0,T^*), W^{s,b})$ for any fractional admissible pair $(a,b)$.
			\item If $T^*<+\infty$, then $\|\psi(t)\|_{\dot{H}^s} \rightarrow \infty$ as $t\uparrow T^*$.
			\item The solution $\psi(t)$ enjoys the following conservations of mass and energy, i.e., for all $t\in [0,T^*)$			
		\end{itemize}
\begin{equation}\label{masscon}
\|\psi(t)\|_{L^2}=\|\psi_0\|_{L^2},
\end{equation}
\begin{equation}\label{energycon}
E(\psi(t))=E(\psi_0 ),
\end{equation}
 where $E(\psi(t))$ is defined by \eqref{energy power} or \eqref{energy Hartree}.
\end{proposition}

In this paper,
we also need the so called Brezis-Lieb's lemma, see \cite{blieb,morjfa}.
\begin{lemma}\label{lemma Brezis Lieb}
 Let $0<p<\infty$. Suppose that $f_n\rightarrow f$ almost everywhere
and $\{f_n\}$ is a bounded sequence in $L^p$, then
\[
\lim_{n\rightarrow \infty}(\|f_n\|_{L^{p}}^{p}-\|f_n-f\|_{L^{p}}^{p}-\|f\|_{L^{p}}^{p})=0.
\]
\begin{align*}
\lim_{n\rightarrow \infty}
\left(\int_{\mathbb{R}^N}(|x|^{-\gamma}\ast|u_n|^2)|u_n|^2dx
-\int_{\mathbb{R}^N}(|x|^{-\gamma}\ast|u_n-u|^2)|u_n-u|^2dx\right)=\int_{\mathbb{R}^N}(|x|^{-\gamma}\ast|u|^2)|u|^2dx.
\end{align*}
\end{lemma}
The following compactness lemma is vital in our discussion, see \cite{dinhfengdcds,f18cpaa}.
\begin{lemma}\label{lemma compactness lemma I}
		Let $N\geq 1$, $0<s<1$, $0<p<\frac{4s}{N-2s}$. Let $\{u_n\}$ be a bounded sequence in $H^s$ and satisfy that
		\[
	 \liminf_{n\rightarrow \infty} \|u_n\|_{L^{p+2}} \geq m,
		\]
for some $m>0$.
		Then there exist a sequence $(x_n)_{n\geq 1}$ in $\mathbb{R}^N$ and $U \in H^s\setminus \{0\}$ such that up to a subsequence,
		\[
		u_n(\cdot+x_n) \rightharpoonup U \text{ weakly in } H^s.
		\]
	\end{lemma}

Next,
we recall a sharp Gagliardo-Nirenberg type inequality established in \cite{bou,zjee}.
\begin{lemma}\label{gn1}
Let $N\geq 2$, $0<s<1$ and $0<p<\frac{4s}{N-2s}$. Then, for all $u\in H^s$,
\begin{equation}\label{gn inequality power}
\|u\|_{L^{p+2}}^{p+2}\leq C_{opt}\|(-\Delta)^{\frac{s}{2}}u\|^{\frac{p N}{2s}}_{L^{2}}
\| u \|^{(p+2)-\frac{pN}{2s}}_{L^{2}},
\end{equation}
where the optimal constant $C_{opt}$ given by
\[
 C_{opt}=\left(\frac{2s(p+2)-p N}{p N}\right)^{\frac{Np}{4s}}\frac{2s(p+2)}{(2s(p+2)-p N)\| R \|_{L^{2}}^{p}},
\]
and $R$ is a ground state solution of the following elliptic equation
\begin{equation}\label{ell21}
(-\Delta)^{s}R+R=|R|^{p}R~~~in~~\mathbb{R}^N.
\end{equation}
In particular, in the $L^2$-critical case $p=\frac{4s}{N}$, $ C_{opt}=\frac{2s+N}{N\| R \|_{L^{2}}^{p}}$.

%Moreover, the following Pohozaev's identities hold true:
%		\begin{align}\label{pohozaev identities}
%		\|R\|^2_{\dot{H}^s} = \frac{Np}{2s(p+2)} \|R\|^{p+2}_{L^{p+2}} = \frac{Np}{4s-(N-2s)p} \|R\|^2_{L^2}.
%		\end{align}
\end{lemma}
\begin{lemma}\cite{zhu16jde}
Let $2s<\gamma<\min\{N,4s\}$ and $N\geq 2$. Then for any $u\in H^s$,
\begin{equation}\label{gn inequality hartree}
\int_{\mathbb{R}^N}\int_{\mathbb{R}^N}\frac{|u(x)|^2|u(y)|^2}{(x-y)^\gamma}dxdy
\leq\left(\frac{4s-\gamma}{\gamma}\right)^{\frac{\gamma}{2s}}\frac{4s}{(4s-\gamma)
\|R\|^2_{L^2}}\|u\|^{\frac{4s-\gamma}{s}}_{L^2}\|u\|^{\frac{\gamma}{s}}_{\dot{H}^s}
\end{equation}
where $R$ is ground state solution of
\[
(-\Delta)^sR+ R-(\frac{1}{|x|^\gamma}\ast |R|^2)R=0.
\]
\end{lemma}
\begin{lemma}\cite{bou,zhu16jde}(Pohozaev identity)\label{lemma Pohozaev identity}
Let $f(u)=|u|^{p}u$ with $\frac{4s}{N}<p<\frac{4s}{N-2s}$ or $f(u)=(|x|^{-\gamma}\ast|u|^2)u$ with $2s<\gamma<\min\{N,4s\}$, and $u\in H^s$ is a weak solution of problem \eqref{elliptic equation}, then
\[
(N-2s)\|u\|^2_{\dot{H}^s}+N\omega \|u\|_{L^2}^2=\frac{2N}{p+2}\|u\|^{p+2}_{L^{p+2}}~~~~~~~~~~~~~~~~~~~~if~f(u)=|u|^{p}u,
\]
\[
(N-2s)\|u\|^2_{\dot{H}^s}+N\omega \|u\|_{L^2}^2=\frac{2N-\gamma}{2}\int_{\mathbb{R}^N}(|x|^{-\gamma}\ast|u|^2)(x)|u|^2(x)dx~~~if~f(u)
=(|x|^{-\gamma}\ast|u|^2)u.
\]
\end{lemma}

Finally,
 we recall the localized virial estimate related to \eqref{e} with $f(u)=|u|^{p}u$, which is the main ingredient in the proof
of the sharp threshold of global
existence and blow-up. The localized virial estimate was used by
Boulenger-Himmelsbach-Lenzmann \cite{bou} to show the existence of finite time blow-up radial solutions to
\eqref{e} in the $L^2$-critical and $L^2$-supercritical cases. Let us start with the following estimate.
	\begin{lemma} [\cite{bou}] \label{lemma various estimate 1}
		Let $N\geq 1$ and $\varphi:\mathbb{R}^N \rightarrow \mathbb{R}$ be such that $\nabla \varphi \in W^{1,\infty}$. Then for all $u \in H^{1/2}$,
		\[
		\left| \int_{\mathbb{R}^N} \overline{u}(x) \nabla \varphi(x) \cdot \nabla u(x) dx \right| \leq C \left( \||\nabla|^{1/2} u\|^2_{L^2} + \|u\|_{L^2} \||\nabla|^{1/2} u\|_{L^2}\right),
		\]
		for some $C>0$ depending only on $\|\nabla \varphi\|_{W^{1,\infty}}$ and $N$.
	\end{lemma}

Let $N\geq 1$, $1/2 \leq s <1$ and $\varphi: \mathbb{R}^N \rightarrow \mathbb{R}$ be such that $\nabla \varphi \in W^{3,\infty}$. Assume $\psi \in C([0,T^*), H^s)$ is a solution to $(\ref{e})$. The localized virial action of $\psi$ is defined by
	\begin{align}\label{virial action}
	M_{\varphi}(\psi(t)):= 2 \int_{\mathbb{R}^N} \nabla \varphi(x) \cdot Im{(\overline{\psi}(t,x) \nabla \psi(t,x))} dx.
	\end{align}
It follows from Lemma $\ref{lemma various estimate 1}$ that $M_\varphi(\psi(t))$ is well-defined. Indeed, by Lemma $\ref{lemma various estimate 1}$,
	\[
	|M_\varphi(\psi(t))| \lesssim C\left(\|\nabla \varphi\|_{L^\infty}, \|\Delta \varphi\|_{L^\infty} \right) \|\psi(t)\|^2_{H^{1/2}} \lesssim C(\varphi) \|\psi(t)\|^2_{H^s} <\infty.
	\]
	To study the time evolution of $M_\varphi(\psi(t))$, we need the following auxiliary function
	\begin{align}
	\psi_m(t,x):= c_s \frac{1}{-\Delta +m} \psi(t,x) = c_s \mathcal{F}^{-1} \left( \frac{ \hat{\psi}(t, \xi)}{|\xi|^2+m}\right), \quad m>0, \label{auxiliary function}
	\end{align}
	where
\[
	c_s:= \sqrt{\frac{\sin \pi s}{\pi}}.
	\]
	Remark that since $\psi(t) \in H^s$, the smoothing property of $(-\Delta+m)^{-1}$ implies that $\psi_m(t) \in H^{s+2}$ for any $t\in [0,T^*)$.
	\begin{lemma} [\cite{bou}] \label{lemma virial action}
		For any $t\in [0,T^*)$, the following identity holds true
		\begin{align}\label{time evolution virial action}
		\frac{d}{dt} M_\varphi(\psi(t)) &= - \int_0^\infty m^s \int_{\mathbb{R}^N} \Delta^2 \varphi |\psi_m(t)|^2 dxdm + 4 \sum_{j,k=1}^N \int_0^\infty m^s \int _{\mathbb{R}^N} \partial^2_{jk} \varphi \partial_j \overline{\psi}_m(t)  \partial_k \psi_m(t) dx dm  \nonumber \\
		&\mathrel{\phantom{=}}-\frac{2 p}{p+2} \int_{\mathbb{R}^N} \Delta \varphi |\psi(t)|^{p+2} dx,
		\end{align}
		where $\psi_m$ is defined in $(\ref{auxiliary function})$.
	\end{lemma}

Using Plancherel's and Fubini's theorem, it follows that
	\begin{align}\label{property u_m}
	\int_0^\infty m^s \int_{\mathbb{R}^N} |\nabla \psi_m| dx dm &= \int_{\mathbb{R}^N} \left(\frac{\sin\pi s}{ \pi } \int_0^\infty \frac{m^s dm}{(|\xi|^2+m)^2} \right) |\xi|^2 |\hat{\psi}(\xi)|^2 d\xi  \nonumber \\
	&= \int_{\mathbb{R}^N} (s |\xi|^{2s-2}) |\xi|^2 |\hat{\psi}(\xi)|^2 d\xi = s \|\psi\|^2_{\dot{H}^s}.
	\end{align}
	If we make formal substitution and take the unbounded function $\nabla \varphi(x) = 2x$, then we have $\partial^2_{jk} \varphi =2\delta_{jk}$ and $\Delta^2\varphi=0$. Using $(\ref{property u_m})$, we find formally the virial identity
	\begin{align}
	\frac{d}{dt} M_{|x|^2}(\psi(t)) &= 8s \|\psi(t)\|^2_{\dot{H}^s} -\frac{4N p}{p+2} \|\psi(t)\|^{p+2}_{L^{p+2}} \nonumber \\
	&= 4Np E(\psi(t)) - 2(Np-4s) \|\psi(t)\|^2_{\dot{H}^s}. \label{virial identity}
	\end{align}
	Now let $\varphi:\mathbb{R}^N \rightarrow \mathbb{R}$ be as above. We assume in addition that $\varphi$ is radially symmetric and satisfies
	\begin{align}
	\varphi(r):= \left\{
	\begin{array}{c l}
	r^2 &\text{for } r \leq 1, \\
	\text{const.} &\text{for } r \geq 10,
	\end{array}
	\right. \quad \text{and } \varphi'' (r) \leq 2 \text{ for } r \geq 0. \label{choice of varphi}
	\end{align}
	Here the precise constant is not important. For $R>0$ given, we define the rescaled function $\varphi_R: \mathbb{R}^N \rightarrow \mathbb{R}$ by
	\begin{align}
	\varphi_R(x) = \varphi_R(r):= R^2 \varphi(r/R). \label{define varphi_R}
	\end{align}
	It is easy to see that
	\begin{align}
	2-\varphi''_R(r) \geq 0, \quad 2-\frac{\varphi'_R(r)}{r} \geq 0, \quad 2N-\Delta \varphi_R(x) \geq 0, \quad \forall r \geq 0, \forall x \in \mathbb{R}^N. \label{property varphi_R}
	\end{align}

Moreover,
	\[
	\|\nabla^j \varphi_R\|_{L^\infty} \lesssim R^{2-j}, \quad j=0, \cdots, 4,
	\]
	and
	\[
	\text{supp}(\nabla^j \varphi_R) \subset \left\{
	\begin{array}{cl}
	\{ |x| \leq 10 R\} &\text{for } j=1,2, \\
	\{ R \leq |x| \leq 10 R \} &\text{for } j=3, 4.
	\end{array}
	\right.
	\]
Finally, we recall the following virial estimate for the time evolution of $M_{\varphi_R}(\psi(t))$, see \cite{bou}.
	\begin{lemma} [$H^s$ radial virial estimate] \label{lemma Hs radial virial estimate}
		Let $N\geq 2$, $\frac{N}{2N-1} \leq s <1$, $0<p < \frac{4s}{N-2s}$, $\varphi_R$ be as in $(\ref{define varphi_R})$ and $\psi \in C([0,T^*), H^s)$ be a radial solution to $(\ref{e})$. Then for any $t\in [0,T^*)$,
		\begin{align}
		\begin{aligned}
		\frac{d}{dt} M_{\varphi_R}(\psi(t)) &\leq 4s \|\psi(t)\|^2_{\dot{H}^s} - \frac{2N p}{p+2} \|\psi(t)\|^{p+2}_{L^{p+2}}  \\
		&\mathrel{\phantom{\leq }} + O \left( R^{-2s} + R^{-\frac{p(N-1)}{2} + \varepsilon s} \|\psi(t)\|^{\frac{p}{2s}+\varepsilon}_{\dot{H}^s} \right)  \\
		&= 2Np E(\psi(t)) -(Np -4s) \|\psi(t)\|^2_{\dot{H}^s} \\
		&\mathrel{\phantom{\leq}} + O \left( R^{-2s} + R^{-\frac{p(N-1)}{2} + \varepsilon s} \|\psi(t)\|^{\frac{p}{2s}+\varepsilon}_{\dot{H}^s} \right),
		\end{aligned}
		\label{Hs radial virial estimate}
		\end{align}
		for any $0<\varepsilon<\frac{(2N-1) p}{2s}$. Here the implicit constant depends only on $\|\psi_0\|_{L^2}, N, \varepsilon, s$ and $p$.
	\end{lemma}

\section{Existence of normalized ground states}
In this section, we will prove Theorem \ref{theorem existence}. Firstly, we establish some preliminaries.

%\begin{lemma}
%Let $f(u)=|u|^{p}u$ with $\frac{4s}{N}<p<\frac{4s}{N-2s}$ or $f(u)=(|x|^{-\gamma}\ast|u|^2)u$ with $2s<\gamma<\min\{N,4s\}$. Then for any $u\in S(c)$, $u^\lambda\in S(c)$, $\|u^\lambda\|^2_{\dot{H}^s}\rightarrow\infty$ and $E(u^\lambda)\rightarrow -\infty$ as $\lambda\rightarrow \infty$, where $u^\lambda(x)=\lambda^{N/2}u(\lambda x)$.
%\end{lemma}
\begin{lemma}\label{lemma max energy}
Let $u^\lambda(x)=\lambda^{N/2}u(\lambda x)$, $f(u)=|u|^{p}u$ with $\frac{4s}{N}<p<\frac{4s}{N-2s}$ or $f(u)=(|x|^{-\gamma}\ast|u|^2)u$ with $2s<\gamma<\min\{N,4s\}$. Then for any $u\in S(c)$, there exists a unique $\lambda_0 >0$ such that

(1) when $f(u)=|u|^{p}u$,
\[
E(u^{\lambda_0}):=\max_{t>0}E(u^\lambda)=
\left(\frac{Np-4s}{2Np}\right)\left(\frac{2s(p+2)}{Np}\right)^{\frac{4s}{Np-4s}}
\frac{\|u\|^{\frac{2Np}{Np-4s}}_{\dot{H}^s}}{\|u\|_{L^{p+2}}^{\frac{4s(p+2)}{Np-4s}}},
\]\\

(2) when $f(u)=(|x|^{-\gamma}\ast|u|^2)u$,
\[
E(u^{\lambda_0}):=\max_{t>0}E(u^\lambda)=
\frac{\gamma-2s}{2\gamma}\frac{(4\gamma)^\frac{\gamma}{\gamma-2s}}{\gamma^{\frac{2s}{\gamma-2s}}}
\frac{\|u\|^{\frac{2\gamma}{\gamma-2s}}_{\dot{H}^s}}{(\int_{\mathbb{R}^N} \left(|x|^{-\gamma}\ast|u|^2)(x)|u(x)|^2 dx\right)^{\frac{2s}{\gamma-2s}}},
\]
and $u^{\lambda_0}\in V(c)$. In particular\\
(i) $\lambda_0 <1$ $\Leftrightarrow$ $Q(u)<0$;\\
(ii) $\lambda_0 =1$ $\Leftrightarrow$ $Q(u)=0$;\\
(iii) $\lambda<\lambda_0$ $\Leftrightarrow$ $Q(u^{\lambda})>0$;\\
(iv) $\lambda>\lambda_0$ $\Leftrightarrow$ $Q(u^{\lambda})<0$;\\
where $Q(u)$ is given in \eqref{qupower} or \eqref{quhartree}.
\end{lemma}
\begin{proof}
We only prove the case $f(u)=|u|^{p}u$. The case $f(u)=(|x|^{-\gamma}\ast|u|^2)u$ is similar. Firstly, we define
\[
g(\lambda):=E(u^\lambda)=\frac{\lambda^{2s}}{2}\|u\|^2_{\dot{H}^s}
-\frac{\lambda^{\frac{Np}{2}}}{p+2}
\int_{\mathbb{R}^N} |u(x)|^{p+2} dx.
\]
Then, $g(\lambda)>0$ for sufficiently small $\lambda>0$ and $g(\lambda)\rightarrow -\infty$ as $\lambda\rightarrow \infty$.
This implies that $g(\lambda)$ has a unique critical point $\lambda_0 >0$ corresponding to its maximum on $(0,\infty)$, and
\[
E(u^{\lambda_0})=\max_{t>0}E(u^\lambda),~~ g'(\lambda_0)  =s\lambda^{2s-1}_0\|u\|^2_{\dot{H}^s}
-\frac{Np}{2(p+2)}\lambda^{\frac{Np-2}{2}}_0\int_{\mathbb{R}^N} |u(x)|^{p+2} dx=0,
\]
which yields
$$Q(u^{\lambda_0})=s \lambda_0^{2s}\|u\|^2_{\dot{H}^s}-\frac{Np}{2(p+2)}
\lambda^{\frac{Np}{2}}_0\int_{\mathbb{R}^N} |u(x)|^{p+2} dx=0.$$
We consequently obtain $u^{\lambda_0}\in V(c)$. Moreover,
\[
Q(u)=s\|u\|^2_{\dot{H}^s}-\frac{Np}{2(p+2)}\int_{\mathbb{R}^N} |u(x)|^{p+2} dx=s\|u\|^2_{\dot{H}^s}(1-\lambda^{\frac{4s-Np}{2}}_0),
\]
which concludes (i) and (ii). (iii) and (iv) follow from the fact that $g'(\lambda)=\lambda^{-1}Q(u^\lambda)$.
\end{proof}

\begin{lemma}\label{w=0}
Let $f(u)=|u|^{p}u$ with $\frac{4s}{N}<p<\frac{4s}{N-2s}$ or $f(u)=(|x|^{-\gamma}\ast|u|^2)u$ with $2s<\gamma<\min\{N,4s\}$. If $u\in H^s$ is a weak solution of problem \eqref{elliptic equation}, then $Q(u)=0$. Moreover, $u=0$ if $\omega\leq 0$.
\end{lemma}
\begin{proof}
When $f(u)=|u|^{p}u$,
by Lemma \ref{lemma Pohozaev identity}, the following Pohozaev identity holds for $u\in H^s$,
\[
(N-2s)\|u\|^2_{\dot{H}^s}+N\omega \|u\|_{L^2}^2=\frac{2N}{p+2}\|u\|^{p+2}_{L^{p+2}}.
\]
Multiplying \eqref{elliptic equation} by $u$ and integrating over $\mathbb{R}^N$, we derive a second identity
\[
\|u\|^2_{\dot{H}^s}+ \omega\|u\|_{L^2}^2=\|u\|^{p+2}_{L^{p+2}}.
\]
Thus we have immediately
\[
Q(u)=\|u\|^2_{\dot{H}^s}-\frac{Np}{2s(p+2)}\|u\|^{p+2}_{L^{p+2}}=0.
\]
Also after simple calculations, we obtain
\[
\omega\|u\|_{L^2}^2=(\frac{2s(p+2)}{Np}-1)\|u\|^2_{\dot{H}^s}.
\]\\
(1) If $\omega<0$, we get $u\equiv 0$ immediately;\\
(2) If $\omega=0$, $\|u\|^2_{\dot{H}^s}=0$ then $u\equiv 0$.\\
The proof for $f(u)=(|x|^{-\gamma}\ast|u|^2)u$ is similar, so we omit the details.
\end{proof}

\begin{lemma}\label{critical point}
Let $f(u)=|u|^{p}u$ with $\frac{4s}{N}<p<\frac{4s}{N-2s}$ or $f(u)=(|x|^{-\gamma}\ast|u|^2)u$ with $2s<\gamma<\min\{N,4s\}$. If $u$ is a critical point of $E|_{S(c)}$, then $E'(u)+\omega_c u=0$ in $H^{-s}$ for some $\omega_c>0$.
\end{lemma}
\begin{proof}
Since $u$ is a critical point of $E|_{S(c)}$, there exists $\omega_c\in \mathbb{R}$ such that $E'(u)+\omega_cu=0$ in $H^{-s}$. Thus
\begin{align}\label{2.1}
\langle E'(u)+\omega_cu,u\rangle=\|u\|^2_{\dot{H}^s}+ \omega_c\|u\|_{L^2}^2-\|u\|^{p+2}_{L^{p+2}}=0.
\end{align}
By Lemma \ref{lemma Pohozaev identity}, $u$ satisfies
\begin{align}\label{p1}
(N-2s)\|u\|^2_{\dot{H}^s}+N\omega \|u\|_{L^2}^2=\frac{2N}{p+2}\|u\|^{p+2}_{L^{p+2}}.
\end{align}
Combining \eqref{2.1} with \eqref{p1}, we have
\begin{align*}
\omega_c=\frac{(4s+2ps-pN)\|u\|^2_{\dot{H}^s}}{Npc}>0
\end{align*}
for $\frac{4s}{N}<p<\frac{4s}{N-2s}$. The proof for $f(u)=(|x|^{-\gamma}\ast|u|^2)u$ is similar, so we omit the details.
\end{proof}

Next, we analyze the property of the function $c\rightarrow m(c)$.
\begin{lemma}\label{E inf}
Let $f(u)=|u|^{p}u$ with $\frac{4s}{N}<p<\frac{4s}{N-2s}$ or $f(u)=(|x|^{-\gamma}\ast|u|^2)u$ with $2s<\gamma<\min\{N,4s\}$. Then
\[
m(c)=\inf_{u\in S(c)}\max_{t>0} E(u^\lambda),
\]
where $u^\lambda=\lambda^{N/2}u(\lambda x)$.
\end{lemma}
\begin{proof}
Firstly, we notice that the minimizing problem in \eqref{minimization problem submanifold} is well-defined. Indeed, when $f(u)=|u|^{p}u$ and $u\in S(c)$, we have
\[
E(u)=E(u)-\frac{2}{Np}Q(u)=
\frac{Np-4s}{2Np}\|u\|_{\dot{H}^s}^2>0.
\]
When $f(u)=(|x|^{-\gamma}\ast|u|^2)u$ and $u\in S(c)$, it follows that
\[
E(u)=E(u)-\frac{1}{\gamma}Q(u)=
\frac{\gamma-2s}{2\gamma}\|u\|_{\dot{H}^s}^2>0.
\]
Thus, we denote $m(c):=\inf_{u\in V(c)}E(u)$.
By Lemma \ref{lemma max energy}, for any $u\in V(c)$,
\begin{align*}
E(u)=\max_{\lambda>0}E(u^\lambda)\geq \inf_{v\in S(c)}\max_{\lambda>0}E(u^\lambda),
\end{align*}
then
\begin{align*}
\inf_{u\in V(c)}E(u)\geq \inf_{u\in S(c)}\max_{\lambda>0}E(u^\lambda).
\end{align*}
On the other hand, by Lemma \ref{lemma max energy}, for any $u\in S(c)$, there exists a unique $\lambda_0>0$ such that $u^{\lambda_0}\in V(c)$ and
\begin{align*}
\max_{\lambda>0}E(u^\lambda)=E(u^{\lambda_0})\geq \inf_{u\in V(c)}E(u).
\end{align*}
This implies that
\begin{align*}
\inf_{u\in S(c)}\max_{\lambda>0}E(u^\lambda)\geq \inf_{u\in V(c)}E(u).
\end{align*}
Thus, we have $m(c)=\inf_{u\in S(c)}\max_{t>0} E(u^\lambda)$.
The proof for $f(u)=(|x|^{-\gamma}\ast|u|^2)u$ is similar, so we omit the details.
\end{proof}
\begin{lemma}\label{c=m(c)}
Let $f(u)=|u|^{p}u$ with $\frac{4s}{N}<p<\frac{4s}{N-2s}$ or $f(u)=(|x|^{-\gamma}\ast|u|^2)u$ with $2s<\gamma<\min\{N,4s\}$. Then
the function $c\rightarrow m(c)$ is strictly decreasing on $(0,\infty)$.
\end{lemma}

\begin{proof}
When $f(u)=|u|^{p}u$, for any $0<c_1<c_2<+\infty$, there exists $u_1\in S(c_1)$ such that
\begin{align*}
\max_{\lambda>0}E(u_1^\lambda)<\theta^{\frac{2ps-Np+4s}{Np-4s}} m(c_1),
\end{align*}
where $\theta=\frac{c_2}{c_1}$.
Set
\[
u_2(x)=\theta^{\frac{2s-N}{4s}}u_1(\theta^{-\frac{1}{2s}}x),
\]
then
\[
\|u_2\|^2_{\dot{H}^s}=\|u_1\|^2_{\dot{H}^s}~and~\|u_2\|^2_{L^2}=\theta\|u_1\|^{2}_{L^{2}}=c_2
,\]
\[
\|u_2\|^{p+2}_{L^{p+2}}=\theta^{\frac{2ps+4s-Np}{4s}}\|u_1\|^{p+2}_{L^{p+2}}.
\]
By Lemma \ref{E inf}, we have
\begin{align*}
m(c_2)&\leq \max_{\lambda>0}E(u^\lambda_2)\\
&=\left(\frac{Np-4s}{2Np}\right)\left(\frac{2s(p+2)}{Np}\right)^{\frac{4s}{Np-4s}}
\frac{\|u_2\|^{\frac{2Np}{Np-4s}}_{\dot{H}^s}}{\|u_2\|^{\frac{4s(p+2)}{Np-4s}}_{L^{p+2}}}\\
&=\left(\frac{Np-4s}{2Np}\right)\left(\frac{2s(p+2)}{Np}\right)^{\frac{4s}{Np-4s}}\theta^{\frac{Np-4s-2ps}{Np-4s}}
\frac{\|u_1\|^{\frac{2Np}{Np-4s}}_{\dot{H}^s}}{\|u_1\|^{\frac{4s(p+2)}{Np-4s}}_{L^{p+2}}}\\
&=\theta^{\frac{Np-4s-2ps}{Np-4s}}\max_{\lambda>0}E(u_1^\lambda)\\
&<\theta^{\frac{Np-4s-2ps}{Np-4s}}\theta^{\frac{2ps-Np+4s}{Np-4s}}m(c_1)\\
&=m(c_1)
\end{align*}
holds for $0<c_1,c_2<+\infty$. The proof for $f(u)=(|x|^{-\gamma}\ast|u|^2)u$ is similar, so we omit the details.
\end{proof}

Now, we solve the minimization problem \eqref{minimization problem submanifold}. To this end, we consider the following minimization problem:
\begin{equation}\label{minimization1}
\widetilde{m}(c)=\inf\{\widetilde{E}(v):~v\in S(c),~Q(v)\leq 0\},
\end{equation}
where
\begin{equation}\label{s1u}
\widetilde{E}(v):=E(v)-\frac{2}{Np}Q(v)=
\frac{Np-4s}{2Np}\|v\|_{\dot{H}^s}^2,~~if~f(u)=|u|^pu,
\end{equation}
\begin{equation}\label{s1u}
\widetilde{E}(v):=E(v)-\frac{1}{\gamma}Q(v)=
\frac{\gamma-2s}{2\gamma}\|v\|_{\dot{H}^s}^2,~~if~f(u)=(|x|^{-\gamma}\ast|u|^2)u.
\end{equation}
%In fact, it follows that
%\begin{equation}\label{minimization problemx}
%\widetilde{m}(c)=\inf \{\widetilde{E}(v):~v\in S(c),~Q(v)= 0\}.
%\end{equation}
%Indeed, if $f(u)=|u|^pu$ and $Q(v)<0$, then we have
%\[
%Q(v^\lambda)=\lambda^{2s}s\|v\|_{\dot{H}^s}^2-\frac{Np}{2(p+2)}\lambda^{\frac{Np}{2}}\|v\|_{L^{p+2}}^{p+2}>0
%\]
%for sufficiently small $\lambda>0$. Thus, there exists $\lambda_0\in (0,1)$ such that
%$Q(v^{\lambda_0})=0$. Moreover, it follows that
%\[
%\widetilde{E}(v^{\lambda_0})=\frac{Np-4s}{2Np}\|v\|_{\dot{H}^s}^2\lambda_0^{2s}
%<\frac{Np-4s}{2Np}\|v\|_{\dot{H}^s}^2.
%\]
%This implies that \eqref{minimization problemx} holds.
\begin{proposition}\label{proposition existence}
Let $N\geq 2$, $f(u)=|u|^{p}u$ with $\frac{4s}{N}<p<\frac{4s}{N-2s}$ or $f(u)=(|x|^{-\gamma}\ast|u|^2)u$ with $2s<\gamma<\min\{N,4s\}$.
 Then there exists $u\in V(c)$ and $\widetilde{E}(u)=\widetilde{m}(c)$.
\end{proposition}
\begin{proof}
We only prove this result for $f(u)=|u|^pu$.
We first show that $\widetilde{m}(c)>0$.
By $Q(v)\leq 0$, we have
\[
\|v\|_{\dot{H}^s}^2\leq \|v\|^{p+2}_{L^{p+2}}\leq C_{opt}\|v\|_{\dot{H}^s}^{\frac{p N}{2s}}
\| v \|^{(p+2)-\frac{pN}{2s}}_{L^{2}},
\]
which implies that
\[
c^{\frac{pN}{2s}-(p+2)} \leq C_{opt}\|v\|_{\dot{H}^s}^{\frac{p N}{2s}-2}.
\]
Taking the infimum over $v$, we get $\widetilde{m}(c)>0$.

We now show the minimizing problem \eqref{minimization1} is attained.
Let $\{v_n\}$ be a minimizing sequence for \eqref{minimization1}, i.e.,
 $\{v_n\}\subseteq S(c)$, $Q(v_n)\leq 0$ and $\widetilde{E}(v_n)\rightarrow \widetilde{m}(c)$ as $n\rightarrow \infty$. Thus, there exists $C_0>0$ such that
\[
\frac{Np}{2(p+2)}\liminf_{n\rightarrow \infty}\|v_n\|^{p+2}_{L^{p+2}}\geq s\liminf_{n\rightarrow \infty}\|v_n\|_{\dot{H}^s}^2\geq C_0>0.
\]
Applying Lemma \ref{lemma compactness lemma I}, there exist a subsequence, still denoted by $\{v_n\}$ and $u\in H^s\backslash \{0\}$ such that
\[
u_n:=\tau_{x_n}v_n\rightharpoonup u\neq 0~~weakly~in ~H^s,
\]
for some $\{x_n\}\subseteq \mathbb{R}^N$.
Moreover, we deduce from Lemma \ref{lemma Brezis Lieb} that
\begin{equation}\label{bl1}
Q(u_n)-Q(u_n-u)-Q(u)\rightarrow 0,
\end{equation}
\begin{equation}\label{bl2}
\widetilde{E}(u_n)-\widetilde{E}(u_n-u)-\widetilde{E}(u)\rightarrow 0,
\end{equation}
\begin{equation}\label{bl3}
\|u_n\|_{L^{2}}^{2}-\|u_n-u\|_{L^{2}}^2-\|u\|_{L^{2}}^{2}\rightarrow 0.
\end{equation}
Now, we show that $Q(u)\leq 0$ and $\|u\|_{L^{2}}^{2}=c$ by excluding the other possibilities:

(1) If $Q(u)> 0$ and $\|u\|_{L^{2}}^{2}<c$, it follows from \eqref{bl1} and
$Q(u_n)\leq 0$ that $Q(u_n-u)\leq0$ for sufficiently large $n$.
Set $c_1=c-\|u\|_{L^{2}}^{2}$ and $w_n=\sqrt{c_1}\|u_n-u\|_{L^{2}}^{-1}(u_n-u)$, then we have
\[
\|u_n-u\|_{L^{2}}\rightarrow \sqrt{c_1},~~w_n\in S(c_1),~~and~~Q(w_n)\leq 0.
\]
Thus, by the definition of
$\widetilde{m}(c_1)$, it follows that
\[
\widetilde{E}(w_n)\geq \widetilde{m}(c_1)~~and~~\widetilde{E}(u_n-u)\geq \widetilde{m}(c_1).
\]
Applying $\widetilde{m}(c_1)=m(c_1)> m(c)$, \eqref{bl2}, we can obtain
$\widetilde{E}(u)=\frac{Np-4s}{2Np}\|u\|_{\dot{H}^s}^2\leq0$
which is a contradiction with $u\neq 0$.

(2) If $Q(u)> 0$ and $\|u\|_{L^{2}}^{2}=c$, then $u_n\rightarrow u$ in $L^2$ as $n\rightarrow \infty$. This implies that
$u_n\rightarrow u$ in $L^{p+2}$ as $n\rightarrow \infty$. On the other hand, we deduce from $Q(u)> 0$ that $Q(u_n-u)\leq0$ for sufficiently large $n$. Thus, we can obtain $u_n\rightarrow u$ in $\dot{H}^s$ as $n\rightarrow \infty$.
This yields $Q(u_n-u)\rightarrow 0$ as $n\rightarrow \infty$. Thus, it follows from \eqref{bl1} and $Q(u)> 0$ that $Q(u_n)> 0$ for sufficiently large $n$, which is a contradiction with $Q(u_n)\leq 0$.

(3) If $Q(u)\leq0$ and $\|u\|_{L^{2}}^{2}<c$, then we conclude from \eqref{bl2} and $\widetilde{m}(\|u\|_{L^{2}}^{2})=m(\|u\|_{L^{2}}^{2})> m(c)=\widetilde{m}(c)$ that
$\widetilde{E}(u_n-u)=\frac{Np-4s}{2Np}\|u_n-u\|_{\dot{H}^s}^2<0$, which is a contradiction.

Therefore, we have $Q(u)\leq 0$ and $\|u\|_{L^{2}}^{2}=c$. It follows from the definition of
$\widetilde{m}(c)$ and the weak lower semicontinuity of norm that
\[
\widetilde{m}(c)\leq \widetilde{E}(u)\leq \liminf_{n\rightarrow \infty}\widetilde{E}(u_n)=\widetilde{m}(c).
\]
This yields that
\[
\widetilde{E}(u)=\widetilde{m}(c).
\]

Finally, we show that $Q(u)=0$. Suppose that $Q(u)<0$ and set
\[
f(\lambda):=Q(u^\lambda)=\lambda^{2s}   s\|u\|_{\dot{H}^s}^2- \frac{N  p}{2}\frac{\lambda^{\frac{N  p}{2}}}{p+2} \|u\|^{p+2}_{L^{p+2}},
\]
then $f(\lambda)>0$ for sufficiently small $\lambda>0$ and $f(1)=Q(u)<0$.
 Therefore, there exists $\lambda_0\in (0,1)$ such that $Q(u^{\lambda_0})=0$.
Then, it follows that
\begin{align*}
\widetilde{E}(u^{\lambda_0})=\frac{Np-4s}{2Np} \|u\|_{\dot{H}^s}^2\lambda_0^{2s}<\widetilde{E}(u)=\widetilde{m}(c),
\end{align*}
which contradicts the definition of $\widetilde{m}(c)$. Hence, we have $Q(u)=0$.
\end{proof}

By the fact $\widetilde{m}(c)=m(c)$ and this proposition, we can obtain the following Corollary.
\begin{corollary}\label{corollary existence}
Let $N\geq 2$, $f(u)=|u|^{p}u$ with $\frac{4s}{N}<p<\frac{4s}{N-2s}$ or $f(u)=(|x|^{-\gamma}\ast|u|^2)u$ with $2s<\gamma<\min\{N,4s\}$. Then there exists $u\in V(c)$ and $E(u)=m(c)$.
\end{corollary}

\begin{lemma}\cite{k-cc} \label{lemma zhang}
Let $X$ be a real Banach space, $U\subset X$ be an open set. Suppose that $f, g_1, \cdot\cdot\cdot, g_m:U\rightarrow\mathbb{R}^1$ are $ \mathrm{C}^1$ functions and $x_0\in M$ is such that
$f(x_0) =\inf_{x\in M}$ with
\[
M=\{x\in U : g_i(x)=0,i=1,2,\cdot\cdot\cdot,m\}.
\]
If $\{g'_i(x_0)\}^m_{i=1}$ is linearly independent, then there exists $k_1, \cdot\cdot\cdot, k_m\in\mathbb{R}$ such that
\begin{align*}
f'(x_0)+\sum^m_{i=1} k_i g'_i(x_0)=0.
\end{align*}
\end{lemma}
\begin{lemma}\label{critical point 1}
Let $f(u)=|u|^{p}u$ with $\frac{4s}{N}<p<\frac{4s}{N-2s}$ or $f(u)=(|x|^{-\gamma}\ast|u|^2)u$ with $2s<\gamma<\min\{N,4s\}$. Then each critical point of
$E\mid_{V(c)}$ is a critical point of $E\mid_{S(c)}$.
\end{lemma}
\begin{proof}
We only prove the case for $f(u)=|u|^pu$.
Suppose that $u$ is a critical point of $E\mid_{V(c)}$, then by Lemma \ref{lemma zhang}, we have an alternative: either (i) $Q'(u)$ and $(\|u\|^2_{L^2})'$ are linearly dependent, or (ii)
there exists $\omega_1, \omega_2\in \mathbb{R}$ such that
\begin{align}\label{3.3}
E'(u)+\omega_1 Q'(u)+\omega_2 u=0~~~in~H^{-s}.
\end{align}
If (i) holds, then $u$ satisfies
\begin{align*}
2s (-\Delta)^s u+\omega^*u -\frac{Np}{2}|u|^pu=0~~~in~H^{-s},
\end{align*}
for some $\omega^*\in \mathbb{R}$. Multiplying the above equation by $u$ and integrating, we get
\begin{align*}
2s \|u\|^2_{\dot{H}^s} +\omega^*\|u\|^2_{L^2} -\frac{Np}{2} \|u\|^{p+2}_{L^{p+2}}=0.
\end{align*}
By Pohozaev identity, we derive
\begin{align*}
(N-2s) \|u\|^2_{\dot{H}^s} +\frac{N}{2s}\omega^*\|u\|^2_{L^2} -\frac{N^2p}{2s(p+2)} \|u\|^{p+2}_{L^{p+2}}=0.
\end{align*}
Thus we have
\begin{align*}
2s \|u\|^2_{\dot{H}^s} -\frac{N^2p^2}{4s(p+2)} \|u\|^{p+2}_{L^{p+2}}=0.
\end{align*}
Notice that $Q(u) =0$ and $\frac{4s}{N}<p<\frac{4s}{N-2s}$, then we have immediately $\|u\|^2_{\dot{H}^s}=0$, which is a contradiction with $u\in S(c)$. This implies that (i) does not occur and (ii) is true. It is enough to show that $\omega_1=0$. By \eqref{3.3} we have
\begin{align}\label{3.4}
&\langle E'(u)+\omega_1 Q'(u)+\omega_2 u,u\rangle \nonumber\\
=&(1+2s\omega_1)\|u\|^2_{\dot{H}^s}-(1+\frac{Np}{2}\omega_1)\|u\|^{p+2}_{L^{p+2}}+\omega_2\|u\|^2_{L^2}=0.
\end{align}
By Pohozaev identity corresponding to equation \eqref{3.3},
\begin{align}\label{3.5}
(1+2s\omega_1)(N-2s)\|u\|^2_{\dot{H}^s}-\frac{2N}{p+2}(1+\frac{Np}{2}\omega_1)\|u\|^{p+2}_{L^{p+2}}+N\omega_2\|u\|^2_{L^2}=0.
\end{align}
Combining \eqref{3.4} with \eqref{3.5} we have
\begin{align}\label{3.6}
\|u\|^2_{\dot{H}^s}=\frac{Np(1+\frac{Np}{2}\omega_1)}{2s(p+2)(1+2\omega_1s)}\|u\|^{p+2}_{L^{p+2}}.
\end{align}
Since $u\in V(c)$, $\|u\|^2_{\dot{H}^s}=\frac{Np}{2s(p+2)}\|u\|^{p+2}_{L^{p+2}}$, then by \eqref{3.6} we have $\omega_1=0$.
Finally, by Lemma \ref{critical point}, we get $\omega_2>0$.
\end{proof}

\textbf{Proof of Theorem \ref{theorem existence}.}
By Corollary \ref{corollary existence}, there exists a couple of weak solution $(u_c,\omega_c)\in \mathcal{M}_c\times \mathbb{R}^+$ to problems \eqref{elliptic equation}-\eqref{mass constraint}.
If $v\in S(c)$ satisfies $E'|_{S(c)}(v)=0$, then by Lemma \ref{w=0}-\ref{critical point}, we have $Q(v)=0$, which implies that $v\in V(c)$. Hence, $E(v)\geq E(u_c)$ and $u_c\in S(c) $ is a normalized ground state of problems \eqref{elliptic equation}-\eqref{mass constraint}.

By Lemma \ref{w=0}, we have $Q(u_c)=s\|u_c\|^2_{\dot{H}^s}-\frac{Np}{2(p+2)}\|u_c\|^{p+2}_{L^{p+2}}=0$. Applying the inequality \eqref{gn inequality power}, we have
\[
\frac{2s(p+2)}{Np}\|u_c\|^2_{\dot{H}^s}=
\|u_c\|^{p+2}_{L^{p+2}}\leq C_{opt}\|u_c\|^{\frac{p N}{2s}}_{\dot{H}^s}
\| u_c \|^{(p+2)-\frac{pN}{2s}}_{L^{2}},
\]
then
\[
\|u_c\|^{\frac{Np-4s}{2s}}_{\dot{H}^s}\geq \frac{2s(p+2)}{Np}C_{opt}^{-1} c^{\frac{pN}{2s}-(p+2)}\rightarrow+\infty
\]
as $c\rightarrow 0^+$, i.e. $\|u_c\|^2_{\dot{H}^s}\rightarrow +\infty$ as $c\rightarrow 0^+$. Moreover,
\[
m(c)=E(u_c)=\frac{Np-4s}{2Np}\|u_c\|^2_{\dot{H}^s}\rightarrow+\infty
\]
as $c\rightarrow 0^+$. From equation \eqref{elliptic equation}, we have $\|u_c\|^2_{\dot{H}^s}+\omega_c\|u_c\|^2_{L^{p+2}}=\|u_c\|^{p+2}_{L^{p+2}}$, then
\begin{align*}
\omega_c&=\frac{1}{c}\left(\|u_c\|^{p+2}_{L^{p+2}}-\|u_c\|^2_{\dot{H}^s}\right)\nonumber\\
&=\frac{1}{c}\left(\frac{2s(p+2)}{Np}\|u_c\|^2_{\dot{H}^s}-\|u_c\|^2_{\dot{H}^s}\right)\\
&=\frac{1}{c}\frac{2s(p+2)-Np}{Np}\|u_c\|^2_{\dot{H}^s}\nonumber\\
&\rightarrow +\infty
\end{align*}
as $c\rightarrow 0^+$, for $\frac{4s}{N}<p<\frac{4s}{N-2s}$.

 Next, we consider the case $c\rightarrow +\infty$. Let $u_1\in V(1)$, $\widetilde{u}(x)=c^{-\frac{2s}{Np-4s}}u_1\left(c^{-\frac{p}{Np-4s}}x\right)$. By some simple calculations, we have
\[
\|\widetilde{u}\|^2_{\dot{H}^s}=c^{\frac{Np-2ps-4s}{Np-4s}}\|u_1\|^2_{\dot{H}^s}~and~ \|\widetilde{u}\|^{p+2}_{L^{p+2}}=c^{\frac{Np-2ps-4s}{Np-4s}}\|u_1\|^{p+2}_{L^{p+2}}.
\]
These imply that $Q(\widetilde{u})=0$ and
\begin{align}
E(\widetilde{u})=&\frac{1}{2}\|\widetilde{u}\|^2_{\dot{H}^s}-\frac{1}{p+2}\|\widetilde{u}\|^{p+2}_{L^{p+2}}\nonumber\\
=&\frac{Np-4s}{2Np}\|\widetilde{u}\|^2_{\dot{H}^s}\\
=&\frac{Np-4s}{2Np}c^{\frac{Np-2ps-4s}{Np-4s}}\|u_1\|^2_{\dot{H}^s}\nonumber\\
\rightarrow& 0 \nonumber
\end{align}
as $c\rightarrow +\infty$, for $\frac{4s}{N}<p<\frac{4s}{N-2s}$. Therefore, $0<m(c)=E(u_c)\leq E(\widetilde{u})\rightarrow 0$ as $c \rightarrow +\infty$. So
\[
\|u_c\|^2_{\dot{H}^s}=\frac{2Np}{Np-4s}m(c)\rightarrow 0
\]
and
\[
\omega_c=\frac{1}{c}\frac{2s(p+2)-Np}{Np}\|u_c\|^2_{\dot{H}^s}\rightarrow 0
\]
as $c\rightarrow +\infty$. Thus the proof is completed.

\section{Proof of Theorem \ref{theorem ground states}.}

\textbf{Proof of Theorem \ref{theorem ground states}.}
It is standard that if $\frac{4s}{N}<p<\frac{4s}{N-2s}$ or $2s<\gamma<\min\{N,4s\}$
\begin{align}
d(c):=\inf_{v\in \mathcal{N}_{\omega_c}} J_{\omega_c}(v)
\end{align}
 is attained by a function $\tilde{u}$, which is a ground state solution to problem \eqref{elliptic equation} with $\omega=\omega_c$, where $\mathcal{N}_{\omega_c}=\{v\in H^s:~ \langle J'_{\omega_c}(v),v\rangle=0,~v\neq0\}$. Then, by Lemma \ref{critical point} and Lemma \ref{critical point 1}, $(u_c,\omega_c)$ and $(\tilde{u},\omega_c)$ are two couples of weak solution to problem \eqref{elliptic equation}. By Lemma \ref{w=0}, we have $\omega_c >0$ and $Q(u_c)=Q(\tilde{u})=0$, which implies that
\begin{align}\label{3.10}
\|u_c\|^2_{\dot{H}^s}=\frac{Np}{2s(p+2)}\|u_c\|^{p+2}_{L^{p+2}},
~and~\|\tilde{u}\|^2_{\dot{H}^s}=\frac{Np}{2s(p+2)}\|\tilde{u}\|^{p+2}_{L^{p+2}}.
\end{align}
Hence, we have
\begin{align}\label{3.11}
\omega_c\|u_c\|^2_{L^2}=\frac{2s(p+2)-Np}{Np}\|u_c\|^2_{\dot{H}^s},
~and~\omega_c\|\tilde{u}\|^2_{L^2}=\frac{2s(p+2)-Np}{Np}\|\tilde{u}\|^2_{\dot{H}^s}.
\end{align}
By \eqref{3.10}-\eqref{3.11}, we have
\begin{align*}
J_\omega(u_c)=\frac{1}{2}\|u_c\|^2_{\dot{H}^s}+\frac{\omega_c}{2}\|u_c\|^2_{L^2}-\frac{1}{p+2}
\|u_c\|^{p+2}_{L^{p+2}}=\frac{s}{N}\|u_c\|^2_{\dot{H}^s},~and~J_{\omega_c}(\tilde{u})
=\frac{s}{N}\|\tilde{u}\|^2_{\dot{H}^s},
\end{align*}
and
\begin{align*}
E(u_c)=(\frac{1}{2}-\frac{2s}{Np})\|u_c\|^2_{\dot{H}^s}=\frac{Np-4s}{2ps}J_{\omega_c}(u_c)~and~
E(\tilde{u})=\frac{Np-4s}{2ps}J_{\omega_c}(\tilde{u}).
\end{align*}
Since $\tilde{u}$ is a ground state solution to problem \eqref{elliptic equation} with $\omega=\omega_c$, then $J_{\omega_c}(u_c)\geq J_{\omega_c}(\tilde{u})$. Thus
\begin{align}
m(c)=E(u_c)=\frac{Np-4s}{2ps}J_{\omega_c}(u_c)\geq \frac{Np-4s}{2ps}J_{\omega_c}(\tilde{u})=E(\tilde{u})\geq m(\|\tilde{u}\|^2_{L^2}).
\end{align}
By Lemma \ref{c=m(c)}, we have $\|u_c\|^2_{L^2}\leq\|\tilde{u}\|^2_{L^2}$, then $\|u_c\|^2_{\dot{H}^s}\leq \|\tilde{u}\|^2_{\dot{H}^s}$, which implies that $J_{\omega_c}(u_c)\leq J_{\omega_c}(\tilde{u})$.
So, $J_{\omega_c}(\tilde{u})= J_{\omega_c}(u_c)=d(c)$.
This completes the proof of Theorem \ref{theorem ground states}.

\section{Proof of Theorem \ref{theorem sharp}.}

\textbf{Proof of Theorem \ref{theorem sharp}.}
 Firstly, we show that the set
$\mathcal{A}_c$ and $\mathcal{B}_c$
are not empty. Indeed, for arbitrary but fixed $u\in S(c)$, set $u^\lambda(x)=\lambda^{\frac{N}{2}}u(\lambda x)$. Then we have $u^\lambda \in S(c)$, for all $\lambda>0$; $E(u^\lambda)\rightarrow 0$ as $\lambda \rightarrow0$ and $Q(u^\lambda)>0$ for sufficiently small $\lambda >0$. This proves that $\mathcal{A}_c \neq \emptyset$. In addition, $E(u^\lambda)\rightarrow -\infty$ and $Q(u^\lambda)\rightarrow -\infty$ as $\lambda \rightarrow +\infty$. Thus, $\mathcal{B}_c \neq \emptyset$.

In the following, we will prove that $\mathcal{A}_{\|\psi_0\|_{L^2}^2}$ and $\mathcal{B}_{\|\psi_0\|_{L^2}^2}$ are two invariant manifolds of \eqref{e}.
Let $\psi_0\in \mathcal{A}_{\|\psi_0\|_{L^2}^2}$, by Proposition 2.1, we see that there
exists a unique solution $\psi\in C([0,T^*),H^s)$ with initial data $\psi_0$. We deduce from the conservations of energy that
\begin{equation}\label{asc}
E(\psi(t))=E(\psi_0)<m(\|\psi_0\|_{L^2}^2),
\end{equation}
for any $t\in [0,T^*)$. In addition, by the continuity of the function $t\mapsto Q(\psi(t))$ and
Corollary 3.7, if there exists
$t_0\in[0,T^*)$ so that  $Q(\psi(t_0))= 0$, then $E(\psi(t_0))\geq m(\|\psi_0\|_{L^2}^2)$, which contradicts with \eqref{asc}. Therefore, we have $Q(\psi(t))>0 $ for any $t\in [0,T^*)$. Similarly, we can prove that $ \mathcal{B}_{\|\psi_0\|_{L^2}^2}$ is invariant under the flow of \eqref{e}.

Now, we prove (1). Let us prove (1) by contradiction. If not, there exists $T^*>0$ such that
\begin{align}\label{3.16}
\lim_{t\rightarrow T^*} \|\psi(t)\|_{\dot{H}^s}=+\infty.
\end{align}
Applying the conservation of energy, we have
\[
E(\psi_0)-\frac{2}{Np}Q(\psi(t))=E(\psi(t))-\frac{2}{Np}Q(\psi(t))=(\frac{1}{2}-\frac{2s}{Np})\|\psi(t)\|^2_{\dot{H}^s},
\]
which implies that
\begin{align*}
\lim_{t \rightarrow T^*}Q(\psi(t))=-\infty,
\end{align*}
if \eqref{3.16} happens. Since $Q(\psi_0)>0$, by continuity there exists $t_0\in (0,T^*)$ such that
\[
Q(\psi(t_0))=0~and~E(\psi(t_0))=E(\psi_0)<m(\|\psi_0\|^2_{L^2}).
\]
This contradicts to the fact that $m(\|\psi_0\|^2_{L^2})=\inf_{u\in V(\|\psi_0\|^2_{L^2})}E(u)$.
Thus, if $\psi_0\in \mathcal{A}_{\|\psi_0\|_{L^2}^2}$, then the solution $\psi(t)$ of \eqref{e} exists globally.

Next, we prove (2).
If $\psi_0\in \mathcal{B}_{\|\psi_0\|_{L^2}^2}$,
then $Q(\psi(t))<0 $ for any $t\in [0,T^*)$. We deduce from Proposition 3.6 that
\[
m(\|\psi_0\|_{L^2}^2)=\widetilde{m}(\|\psi_0\|_{L^2}^2)\leq \widetilde{E}(\psi(t))=E(\psi(t))-\frac{2}{Np}Q(\psi(t))< E(\psi_0)-\frac{Q(\psi(t))}{2s},
\]
 for all $t\in [0,T^*)$. This implies that
\begin{align}\label{Qxiaoyuling}
Q(\psi(t))\leq 2s(E(\psi_0)-m(\|\psi_0\|_{L^2}^2))<0,
\end{align}
  for all $t\in [0,T^*)$.

Now, we claim that there exists $C_1>0$ such that
\begin{align}\label{claim2}
\frac{d}{dt} M_{\varphi_R}(\psi(t))\leq -C_1\|\psi(t)\|^2_{\dot{H}^s},
\end{align}
for $f(u)=|u|^pu$ and any $t\in [0,T^*)$, where $M_{\varphi_R}(\psi(t))$ is defined by \eqref{virial action}.
Firstly, we prove that there exists $C_2>0$ such that
\begin{align}\label{lowerbound}
\|\psi(t)\|_{\dot{H}^s}\geq C_2,
\end{align}
for every $t\in [0,T^*)$. Indeed, suppose this bound is not true, then there exists
$\{t_k\}\subseteq [0,T^*)$ such that
$\|\psi(t_k)\|_{\dot{H}^s}\rightarrow 0$. However, we deduce from mass conservation and the sharp
Gagliardo-Nirenberg inequality \eqref{gn inequality power} that
\[
\|\psi(t_k)\|_{L^{p+2}}^{p+2}\leq C_{opt}\|\psi(t_k)\|^{\frac{p N}{2s}}_{\dot{H}^s}
\|\psi(t_k)\|^{(p+2)-\frac{pN}{2s}}_{L^{2}}\rightarrow 0
\]
as $k\rightarrow \infty$.
Therefore, we have
\begin{align*}
Q(\psi(t_k)):=s\|\psi(t_k)\|_{\dot{H}^s}^2 -\frac{Np}{2(p+2)}\|\psi(t_k)\|_{L^{p+2}}^{p+2}\rightarrow 0,
\end{align*}
as $k\rightarrow \infty$, which contradicts to \eqref{Qxiaoyuling}.

We now prove \eqref{claim2}. Since the solution $\psi(t)$ is radial, we apply Lemma
\ref{lemma Hs radial virial estimate} to have
	\begin{align*}
		\frac{d}{dt} M_{\varphi_R}(\psi(t)) &\leq 4s \|\psi(t)\|^2_{\dot{H}^s}  - \frac{2N p}{p+2} \|\psi(t)\|^{p+2}_{L^{p+2}}   \\
		&\mathrel{\phantom{\leq}} + O \left( R^{-2s}+ R^{-\frac{p(N-1)}{2} + \varepsilon s} \|\psi(t)\|^{\frac{p}{2s}+\varepsilon}_{\dot{H}^s} \right),
		\end{align*}
for all  $t\in [0,T^*)$ and $R>1$. Thanks to the assumption $p<4s$, we can apply the Young inequality
to obtain for any $\eta>0$,
\[
R^{-\frac{p(N-1)}{2} + \varepsilon s} \|\psi(t)\|^{\frac{p}{2s}+\varepsilon}_{\dot{H}^s}\leq C\eta \|\psi(t)\|_{\dot{H}^s}^2+\eta^{-\frac{p+2\varepsilon s}{4s-p-2\varepsilon s}}
R^{-\frac{2s(p(N-1)-2\varepsilon s)}{4s-p-2\varepsilon s}}.
\]
We thus obtain
		\begin{align}\label{verilequaty}
		\frac{d}{dt} M_{\varphi_R}(\psi(t)) &\leq  4s \|\psi(t)\|^2_{\dot{H}^s} - \frac{2N p}{p+2} \|\psi(t)\|^{p+2}_{L^{p+2}} + C\eta\|\psi(t)\|^2_{\dot{H}^s}  \nonumber\\
		&\mathrel{\phantom{\leq}}+ O \left( R^{-2s} + \eta^{-\frac{p+2\varepsilon s}{4s-p-2\varepsilon s}}
R^{-\frac{2s(p(N-1)-2\varepsilon s)}{4s-p-2\varepsilon s}} \right),
		\end{align}
for all  $t\in [0,T^*)$, any $\eta>0$, any $R>1$ and some constant $C>0$.

We fix $t\in [0,T^*)$ and denote
\[
\mu:=\frac{4Np|E(\psi_0)|+2}{Np-4s}.
\]
We consider two cases.

\textbf{Case 1.}
\[
\|\psi(t)\|^2_{\dot{H}^s}\leq \mu.
\]
Since
\[
 4s \|\psi(t)\|^2_{\dot{H}^s} - \frac{2N p}{p+2} \|\psi(t)\|^{p+2}_{L^{p+2}}=4Q(\psi(t))\leq 8s(E(\psi_0)-m(\|\psi_0\|_{L^2}^2))
\]
for all  $t\in [0,T^*)$, we have
\begin{align*}
\frac{d}{dt} M_{\varphi_R}(\psi(t)) &\leq  8s(E(\psi_0)-m(\|\psi_0\|_{L^2}^2))+ C\eta\mu  \\
		&\mathrel{\phantom{\leq}}+ O \left( R^{-2s} + \eta^{-\frac{p +2\varepsilon_2s}{4s-p-2\varepsilon s}}
R^{-\frac{2s(p (N-1)-2\varepsilon s)}{4s-p-2\varepsilon s}} \right).
		\end{align*}
By choosing $\eta>0$ small enough and $R>1$ large enough depending on $\eta$, it follows that
\begin{align}\label{case1}
\frac{d}{dt} M_{\varphi_R}(\psi(t)) &\leq  8s(E(\psi_0)-m(\|\psi_0\|_{L^2}^2))\leq \frac{8s(E(\psi_0)-m(\|\psi_0\|_{L^2}^2))}{\mu}\|\psi(t)\|^2_{\dot{H}^s}.
\end{align}

\textbf{Case 2.}
\[
\|\psi(t)\|^2_{\dot{H}^s}> \mu.
\]
In this case, it follows from conservation of energy that
\begin{align*}
 4s \|\psi(t)\|^2_{\dot{H}^s}  - \frac{2N p}{p+2} \|\psi(t)\|^{p+2}_{L^{p+2}}
=&2Np E(\psi(t)) -(Np-4s) \|\psi(t)\|^2_{\dot{H}^s} \\ \leq& \frac{\mu}{2} (Np-4s)-1-(Np-4s) \|\psi(t)\|^2_{\dot{H}^s}.
\end{align*}
We thus obtain
		\begin{align*}
		\frac{d}{dt} M_{\varphi_R}(\psi(t)) &\leq  -1-\frac{Np-4s}{2}\|\psi(t)\|^2_{\dot{H}^s}+ C\eta\|\psi(t)\|^2_{\dot{H}^s} \\
		&\mathrel{\phantom{\leq}}+ O \left( R^{-2s} + \eta^{-\frac{p+2\varepsilon s}{4s-p-2\varepsilon s}}
R^{-\frac{2s(p(N-1)-2\varepsilon s)}{4s-p-2\varepsilon s}} \right).
		\end{align*}
Since $Np -4s>0$, we choose $\eta>0$ small enough so that
\[
\frac{Np-4s}{2}-C\eta\geq \frac{Np-4s}{4}.
\]
We next choose $R>1$ large enough depending on $\eta$ so that
\[
-1+ O \left( R^{-2s} + \eta^{-\frac{p+2\varepsilon s}{4s-p-2\varepsilon s}}
R^{-\frac{2s(p(N-1)-2\varepsilon s)}{4s-p-2\varepsilon s}} \right)\leq 0.
\]
We thus obtain
\begin{align*}
		\frac{d}{dt} M_{\varphi_R}(\psi(t)) \leq -\frac{Np -4s}{4}\|\psi(t)\|^2_{\dot{H}^s}.
		\end{align*}

We are now able to show that the solution $\psi(t)$ blows up in a  finite time. Assume by contradiction
that $T^*=\infty$. It follows from \eqref{claim2} and \eqref{lowerbound} that $\frac{d}{dt}\mathcal{M}_{\varphi}[\psi(t)]\leq -C$ with some constant $C>0$. Integrating this bound, we conclude that $\mathcal{M}_{\varphi}[\psi(t)]<0$ for all $t\geq t_1$ with some time sufficiently large time $t_{1}\gg1$. Thus, integrating \eqref{claim2} on $[t_{1},t]$, we obtain
\begin{equation} \label{302y}
\mathcal{M}_{\varphi}[\psi(t)]\leq-c\int_{t_{1}}^{t}\|(-\Delta)^{\frac{s}{2}}\psi(\tau)\|^{2}_{L^{2}}d\tau\,\,\,\,\mbox{for all}~~~t\geq t_{1}.
\end{equation}
On the other hand, we use Lemma 2.5 and $L^{2}$-mass conservation to find that
\begin{equation} \label{302x}
\mid\mathcal{M}_{\varphi}[\psi(t)]\mid\leq C(\varphi_{R})(\|(-\Delta)^{\frac{s}{2}}\psi(t)\|^{\frac{1}{s}}_{L^{2}}
+\|(-\Delta)^{\frac{s}{2}}\psi(t)\|^{\frac{1}{2s}}_{L^{2}}),
\end{equation}
where we used the interpolation estimate $\| |\nabla|^{\frac{1}{2}} \psi\|_{L^{2}}\leq\| \psi\|_{L^{2}}^{1-\frac{1}{2s}}\|(-\Delta)^{\frac{s}{2}}\psi\|_{L^{2}}^{\frac{1}{2s}}$ for $s>\frac{1}{2}$.

So, we deduce from \eqref{lowerbound} and \eqref{302x} that
 \begin{equation}
 \begin{gathered}
   |\mathcal{M}_{\varphi}[\psi(t)]|\leq C(\varphi_{R})\|(-\Delta)^{\frac{s}{2}}\psi(t)\|^{\frac{1}{s}}_{L^{2}}.
   \end{gathered}
 \label{e3.11}
\end{equation}
This, together with \eqref{302y}, implies that
\begin{equation}
 \begin{gathered}
  \mathcal{M}_{\varphi}[\psi(t)]\leq -C(\varphi_{R})\int_{t_{1}}^{t} |\mathcal{M}_{\varphi}[\psi(\tau)]|^{2s}d\tau\,\,\,\,\,\mbox{for}\,\,\,\,t\geq t_{1}.
   \end{gathered}
 \label{e3.12}
\end{equation}
This yields $\mathcal{M}_{\varphi}[\psi(t)]\leq-C(\varphi_{R})| t-t_{\ast}|^{1-2s}$ for $s>\frac{1}{2}$ with some $t_{\ast}<+\infty$. Therefore, we have $\mathcal{M}_{\varphi}[\psi(t)]\rightarrow-\infty$ as $t\rightarrow t_{\ast}$. Hence the solution $\psi(t)$ cannot exist for all time $t\geq0$ and consequently we must have that $T^*<+\infty$ holds.

Finally, we prove (3). By a similar argument as the case $f(\psi)=|\psi|^{p}\psi$,
we can also obtain \eqref{Qxiaoyuling} for \eqref{e} with $f(\psi)=(|x|^{-\gamma}\ast|\psi|^2)\psi$.
It follows from  \cite{blowup-Har,zhu16jde} that
$x\psi(t)\in L^2$, $x\cdot\nabla \psi(t)\in L^2$ for all $t\in[0,T^*)$. Moreover,
$\int_{\mathbb{R}^N}\bar{\psi}x(-\Delta)^{1-s}x\psi dx$ is non-negative and
\[
\frac{d^2}{dt^2}\int_{\mathbb{R}^N}\bar{\psi}(t,x)x(-\Delta)^{1-s}x\psi(t,x) dx\leq 2Q(\psi(t))\leq 4s(E(\psi_0)-m(\|\psi_0\|_{L^2}^2))<0,
\]
where we use \eqref{Qxiaoyuling}.
This implies that there exists $0<T^*<\infty$ such that $$\int_{\mathbb{R}^N}\bar{\psi}(T^*,x)x(-\Delta)^{1-s}x\psi(T^*,x
) dx=0.$$
Now, using the conservation of mass and the inequality
\[
\|u\|_{L^2}^2\leq\frac{2}{N}
\left(\int_{\mathbb{R}^N}\bar{u}x(-\Delta)^{1-s}xudx\right)^{1/2}
\left(\int_{\mathbb{R}^N}\bar{u}(-\Delta)^{s}u dx\right)^{1/2}, ~~for~u\in H^s,
\]
 we see that for all $t\in[0,T^*)$
\begin{align*}
\|\psi_0\|_{L^2}^2=\|\psi(t)\|_{L^2}^2&\leq\frac{2}{N}
\left(\int_{\mathbb{R}^N}\bar{\psi}(t,x)x(-\Delta)^{1-s}x\psi(t,x)dx\right)^{1/2}
\left(\int_{\mathbb{R}^N}\bar{\psi}(t,x)(-\Delta)^{s}\psi(t,x) dx\right)^{1/2}\\&\leq\frac{2}{N}
\left(\int_{\mathbb{R}^N}\bar{\psi}(t,x)x(-\Delta)^{1-s}x\psi(t,x)dx\right)^{1/2}
\|\psi(t)\|_{\dot{H}^s}.
		\end{align*}
This yields that $\lim_{t\rightarrow T^*}\|\psi(t)\|_{\dot{H}^s}=\infty$.
This completes the proof of Theorem \ref{theorem sharp}.

\vspace{0.5cm}

{\bf Acknowledgments}

The first author is supported by the National Natural Science Foundation of China (No. 11601435). The third author is supported by the National Natural Science Foundation of China (No. 11501395).

\end{document}